\documentclass{amsart}
\usepackage{amssymb,amsmath}
\usepackage{graphicx}
\usepackage{enumerate}
\usepackage{ifthen}
\usepackage{pdfpages}
\usepackage{hyperref}

\DeclareMathOperator{\tr}{tr}

\DeclareMathOperator{\diag}{diag}

\def\R{{\mathbb{R}}}
\def\N{{\mathbb{N}}}

\def\Sc{{\mathcal{S}}}
\def\theta{{\vartheta}}
\def\phi{{\varphi}}
\def\epsilon{{\varepsilon}}
\long\def\umbruch{{\displaybreak[1]}}

\def\eg{e.\,g.\ }
\def\ie{i.\,e.\ }

\newcommand{\A}[1]{\ifthenelse{#1 = 2}{\lvert A\rvert^{#1}}{\tr A^{#1}}}
\renewcommand{\vec}[1]{\mathbf{#1}}

\mathchardef\ordinarycolon\mathcode`\:
  \mathcode`\:=\string"8000
  \begingroup \catcode`\:=\active
    \gdef:{\mathrel{\mathop\ordinarycolon}}
  \endgroup

\newtheorem{theorem}{Theorem}[section]
\newtheorem{lemma}[theorem]{Lemma}

\newtheorem{corollary}[theorem]{Corollary}

\theoremstyle{definition}
\newtheorem{definition}[theorem]{Definition}
\newtheorem{example}[theorem]{Example}

\newtheorem{remark}[theorem]{Remark}

\numberwithin{equation}{section}
\begin{document}
\title[Pinched hypersurfaces]{Pinched hypersurfaces shrink to round points}

\author{Martin Franzen}
\address{Martin Franzen, Universit\"at Konstanz,
   Universit\"atsstrasse 10, D-78464 Konstanz, Germany}
\curraddr{}
\def\ukaddress{@uni-konstanz.de}
\email{Martin.Franzen\ukaddress}

\address{Martin Westerholt-Raum, Max Planck Institute for Mathematics,
   Vivatsgasse 7, D-53111 Bonn, Germany}
\curraddr{}
\def\raumaddress{@raum-brothers.eu}
\email{Martin\raumaddress}

\address{Ferdinand Kuhl, Digital Competence,
   Annakirchstrasse 190, D-41063 M\"onchengladbach, Germany}
\curraddr{}
\def\kuhladdress{@digital-competence.de}
\email{F.Kuhl\kuhladdress}

\thanks{We would like to thank F. Kuhl, B. Lambert, M. Langford, M. Makowski, O. Schn\"urer,}
\thanks{W. Stein, B. Stekeler, and M. Westerholt-Raum for discussions and support.}
\thanks{The author is a member of the DFG priority program SPP 1489.}


\date{February 27, 2015.}

\dedicatory{With an appendix by M. Franzen, M. Westerholt-Raum, F. Kuhl.}


\begin{abstract}
  We investigate the evolution of closed strictly convex hypersurfaces in $\R^{n+1}$, $n=3$,
  for contracting normal velocities, including powers 
  of the mean curvature, $H$, 
  of the norm of the second fundamental form, $|A|$,
  and of the Gauss curvature, $K$.
  We prove convergence to a round point
  for $2$-pinched initial hypersurfaces.
  In $\R^{n+1}$, $n=2$, natural quantities 
  exist for proving convergence to a round point for many normal velocities.
  Here we present their counterparts for arbitrary dimensions $n\in\N$.
\end{abstract}

\maketitle


\tableofcontents 

\section{Overview}

We consider the geometric flow equation
\begin{align}\label{eq3:evol}
  \begin{cases}
    \frac{d}{dt} X=-F\nu,&\\
    X(\cdot,0)=M_0
  \end{cases}
\end{align}
and ask whether closed strictly convex hypersurfaces $M_{0\leq t < T}$ in $\R^{n+1}$, $n=3$, shrink to round points.

For the cubed mean curvature, $F=H^3$, the answer is affirmative if the initial hypersurface $M_0$ is $2$-pinched, \ie
the principal curvatures $\left(\lambda_i\right)_{1\leq i\leq 3}$ fulfill
\begin{align*}
  \frac{\lambda_i}{\lambda_j} \leq 2
\end{align*}
everywhere on $M_0$ for all $1\leq i,j\leq 3$. This is our main Theorem \ref{thm3:H3}. 

Furthermore, we sketch the proof of similar results for the square of the norm of the second fundamental form, $F=|A|^2$, 
and the Gauss curvature, $F=K$.

So far, strong pinching assumptions were needed to show convergence to a round point \cite{am:convex,bc:deforming,fs:convexity}. \\

\textbf{The paper is structured as follows}: 

\textit{$\bullet$ Notation}: 
We give a quick introduction to differential geometric quantities used in this paper, \eg the induced metric, 
the second fundamental form, and the principal curvatures. 

\textit{$\bullet$ Linear operator $L$}: 
We calculate the linear operator $Lw:=\frac{d}{dt}w-F^{ij}w_{;ij}$ for a function $w$ of the principal curvatures $\lambda_i$, $i=1,2,3$,
at a critical point of $w$.
To improve readability, we also choose normal coordinates at that critical point,
\ie $g_{ij}=\delta_{ij}$, and $\left(h_{ij}\right)=\diag\left(\lambda_1,\lambda_2,\lambda_3\right)$.
This lays the groundwork for subsequent calculations.

\textit{$\bullet$ Vanishing functions}:
In $\R^{n+1}$, $n=2$, for many normal velocities $F$ the quantity
\begin{align*}
  \frac{\left(\lambda_1-\lambda_2\right)^2}{\left(\lambda_1\,\lambda_2\right)^2} F^2
\end{align*}
seems to be the natural choice when showing convergence to a round point.
As in \cite{mf:when}, we call this quantity a \textit{vanishing function} for a normal velocity $F$.
It is used by B. Andrews for the Gauss curvature flow \cite{ba:gauss}, 
by F. Schulze and O. Schn\"urer for the $H^\sigma$-flow \cite{fs:convexity}, 
by B. Andrews and X. Chen for the $|A|^\sigma$ and the $\tr A^\sigma$-flow \cite{ac:surfaces}. 

The quantity 
\begin{align*}
  \sum_{i<j} \frac{\left(\lambda_i-\lambda_j\right)^2}{\left(\lambda_i\,\lambda_j\right)^2} F^2
\end{align*}
is the counterpart of a vanishing function for arbitrary dimensions $n\in\N$.
In particular, we work with this quantity in $\R^{n+1}$, $n=3$.

\textit{$\bullet$ $H^3$-flow}:
The proof of our main Theorem \ref{thm3:H3} is based on investigating the quantities
\begin{align*}
  \phi_{H^3} =&\, \frac{(a-b)^2+(a-c)^2+(b-c)^2}{(a+b+c)^2}, \\
  \text{and }\quad \psi_{H^3} =&\,  \left( \frac{(a-b)^2}{(a\,b)^2} + \frac{(a-c)^2}{(a\,c)^2} + \frac{(b-c)^2}{(b\,c)^2} \right) \left(H^3\right)^2,
\end{align*}
which are homogeneous functions of the principal curvatures $a\equiv \lambda_1$, $b\equiv\lambda_2$, and $c\equiv\lambda_3$.
First we show that the estimate $\phi_{H^3}\leq h := 1/8$ is preserved during the $H^3$-flow  
if the initial hypersurface $M_0$ is $2$-pinched. 
Next we prove that $\psi_{H^3}$ is bounded in time on the set where $\phi_{H^3}\leq h$.
This involves the maximum-principle, the linear operator $L$ and our computer program $[\textsc{cp}]$.
Finally, we show convergence to a round point combining the boundedness of $\psi_{H^3}$ and the proof of \cite[Theorem A.1.]{fs:convexity} 
by F. Schulze and O. Schn\"urer. 

\textit{$\bullet$ $|A|^2$-flow and Gauss curvature flow}:
We sketch the proof of results similar to our main Theorem \ref{thm3:H3} for the $|A|^2$-flow, and for the Gauss curvature flow.

\textit{$\bullet$ Appendix}:
Some of the Lemmas leading up to the proof our main Theorem \ref{thm3:H3} rely
on the computer program $[\textsc{cp}]$, where we use a Monte-Carlo method.
For the convenience of the reader, we include the source code of $[\textsc{cp}]$ 
in three different programming languages, namely the computer algebra systems Mathematica, Sage, and Maple.

\section{Acknowledgments}

We would like to thank O. Schn\"urer for suggesting 
the use of two monotone quantities instead of one.
In particular, we thank O. Schn\"urer for proposing the quantity $\psi_{|A|^2}$, 
and M. Makowski for proposing the quantity $\phi_{|A|^2}$.
We are also indebted to M. Westerholt-Raum and F. Kuhl for their help in translating  
the computer program $[\textsc{cp}]$ from Mathematica to Sage and to Maple.

\section{Notation}

For a quick introduction of the standard notation we adopt the 
corresponding chapter from \cite{os:surfacesA2}. 

We use $X=X(x,\,t)$ to denote the embedding vector of an $n$-manifold $M_t$ into $\R^{n+1}$ and 
$\frac{d}{dt} X=\dot{X}$ for its total time derivative. 
It is convenient to identify $M_t$ and its embedding in $\R^{n+1}$.
The normal velocity $F$ is a homogeneous symmetric function of the principal curvatures.
We choose $\nu$ to be the outer unit normal vector to $M_t$. 
The embedding induces a metric $g_{ij} := \langle X_{,i},\, X_{,j} \rangle$ and
the second fundamental form $h_{ij} := -\langle X_{,ij},\,\nu \rangle$ for all $i,\,j = 1,\ldots,n$. 
We write indices preceded by commas to indicate differentiation with respect to space components, 
\eg $X_{,k} = \frac{\partial X}{\partial x_k}$ for all $k=1,\ldots,n$.

We use the Einstein summation notation. 
When an index variable appears twice in a single term it implies summation of that term over all the values of the index.

Indices are raised and lowered with respect to the metric
or its inverse $\left(g^{ij}\right)$, 
\eg $h_{ij} h^{ij} = h_{ij} g^{ik} h_{kl} g^{lj} = h^k_j h^j_k$.

The principal curvatures $\lambda_i$, $i=1,\ldots,n$, are the eigenvalues of the second fundamental
form $\left(h_{ij}\right)$ with respect to the induced metric $\left(g_{ij}\right)$. 
For $n=3$, we name the principle curvatures also $a\equiv \lambda_1$, $b\equiv \lambda_2$, and $c\equiv \lambda_3$. 
A surface is called strictly convex
if all principal curvatures are strictly positive.
We will assume this throughout the paper.
Therefore, we may define the inverse of the second fundamental
form denoted by $(\tilde h^{ij})$.

Symmetric functions of the principal
curvatures are well-defined, we will use the Gauss curvature
$K=\frac{\det h_{ij}}{\det g_{ij}} = \prod^n_{i=1} \lambda_i$, the mean curvature
$H=g^{ij} h_{ij} = \sum^n_{i=1} \lambda_i$, the square of the norm of the second fundamental form
$|A|^2= h^{ij} h_{ij} = \sum^n_{i=1} \lambda_i^2$, and the trace of powers of the second fundamental form 
$\tr A^{\sigma} = \tr \left(h^i_j\right)^{\sigma} = \sum^n_{i=1} \lambda_i^\sigma$. We write indices preceded by semi-colons 
to indicate covariant differentiation with respect to the induced metric,
e.\,g.\ $h_{ij;\,k} = h_{ij,k} - \Gamma^l_{ik} h_{lj} - \Gamma^l_{jk} h_{il}$, 
where $\Gamma^k_{ij} = \frac{1}{2} g^{kl} \left(g_{il,j} + g_{jl,i} - g_{ij,l}\right)$.
It is often convenient to choose normal coordinates, \ie coordinate systems such
that at a point the metric tensor equals the Kronecker delta, $g_{ij}=\delta_{ij}$,
in which $\left(h_{ij}\right)$ is diagonal, $(h_{ij})=\diag(\lambda_1,\ldots,\lambda_n)$. 
Whenever we use this notation, we will also assume that we have 
fixed such a coordinate system. We will only use a Euclidean metric
for $\R^{n+1}$ so that the indices of $h_{ij;\,k}$ commute according to
the Codazzi-Mainardi equations.

A normal velocity $F$ can be considered as a function of principal curvatures $\lambda_i$, $i=1,\ldots,n$,
or $(h_{ij},\,g_{ij})$. We set $F^{ij}=\frac{\partial F}{\partial h_{ij}}$,
$F^{ij,\,kl}=\frac{\partial^2 F}{\partial h_{ij}\partial h_{kl}}$. 
Note that in coordinate
systems with diagonal $h_{ij}$ and $g_{ij}=\delta_{ij}$ as mentioned
above, $F^{ij}$ is diagonal.

\section{Linear operator $L$}

We begin this chapter with Definition \ref{def3:linear operator} of the linear operator $Lw$
for a function $w$ of the principal curvatures $\lambda_i$, $i=1,\ldots,n$.
Then we calculate the linear operator $Lw$ at a critical point of $w$ in $\R^{n+1}$, $n=3$.
To improve readability, we also choose normal coordinates at that critical point,
\ie $g_{ij}=\delta_{ij}$, and $\left(h_{ij}\right)=\diag\left(\lambda_1,\lambda_2,\lambda_3\right)\equiv\diag\left(a,b,c\right)$.
This is Lemma \ref{lem3:critical point}. 
In Corollary \ref{cor3:critical point} we will see that the linear operator $Lw$ has the form
\begin{align*}
  Lw = \vec{C_w} 
     + \vec{E_w}\,\vec{x_0}^2
     + \vec{x_1}^\top\, M^\vec{R_w}\, \vec{x_1} 
     + \vec{x_2}^\top\, M^\vec{S_w}\, \vec{x_2} 
     + \vec{x_3}^\top\, M^\vec{T_w}\, \vec{x_3},
\end{align*}
where $\vec{C_w}(a,b,c), \vec{E_w}(a,b,c)$ are functions in $\R$, and
$M^\vec{R_w}(a,b,c)$, $M^\vec{S_w}(a,b,c)$, $M^\vec{T_w}(a,b,c)$ are functions in $\R^{2\times 2}$ with 
\begin{align*}
  \vec{x_0} = h_{12;3},\;
  \vec{x_1} = \binom{h_{22;1}}{h_{33;1}},\;
  \vec{x_2} = \binom{h_{11;2}}{h_{33;2}},\;
  \vec{x_3} = \binom{h_{11;3}}{h_{22;3}}.
\end{align*}
In subsequent calculations we need the linear operator $Lw$ to be non-positive
for some set $\Sc\subset\R^3_{+}$.
We achieve this by checking the non-positivity 
of each of the functions $\vec{C_w}, \vec{E_w}$, $M^\vec{R_w}$, $M^\vec{S_w}$, 
and $M^\vec{T_w}$ on $\Sc\subset\R^3_{+}$.
In Remark \ref{rem3:sufficient conditions} we state
the criterion we use in our computer program $[\textsc{cp}]$
to determine the negative semi-definiteness of $M^\vec{R_w}$, $M^\vec{S_w}$, 
and $M^\vec{T_w}$.

\begin{definition}[Linear operator]\label{def3:linear operator}
  Let $w$ be a function of the principal curvatures.
  Then we define the linear operator $L$ by
  \begin{align}
    Lw=\frac{d}{dt}w-F^{ij}w_{;ij},
  \end{align}
  which is corresponding to the geometric flow equation \eqref{eq3:evol}.
\end{definition}
 
\begin{lemma}[Linear operator]\label{lem3:linear operator}
  Let $w=w\big(h^j_i\big)$ be a function of the principal curvatures.
  Let $L$ be defined as in \eqref{def3:linear operator}.
  Then we have
  \begin{align}
    \begin{split}
      Lw =&\, w^{ij}\left(h_{ij}F^{kl}h^m_kh_{lm}+h^m_ih_{jm}\left(F-F^{kl}h_{kl}\right)\right) \\
          &\quad+\left(w^{ij}F^{kl,rs}-F^{ij}w^{kl,rs}\right)h_{kl;i}h_{rs;j}.
    \end{split}
  \end{align}
\end{lemma}
\begin{proof}
  We refer to \cite[Lemma 4.5]{mf:when}.
\end{proof}

\begin{lemma}[Second derivatives]\label{lem3:second derivatives}
  Let $f$ be a normal velocity $F$ or a function $w$ of the principal curvatures.
  Then we have
  \begin{align}
    f^{ij,kl}\eta_{ij}\eta_{kl} 
    = \sum_{i,j} \frac{\partial^2 f}{\partial \lambda_i\partial \lambda_j}\eta_{ii}\eta_{jj}
     +\sum_{i\neq j}\frac{\frac{\partial f}{\partial \lambda_i}-\frac{\partial f}{\partial \lambda_j}}{\lambda_i-\lambda_j}\eta_{ij}^2
  \end{align}
  for any symmetric matrix $(\eta_{ij})$ and $\lambda_i\neq\lambda_j$, 
  or $\lambda_i=\lambda_j$ and the last term is interpreted as a limit.
\end{lemma}
\begin{proof}
  We refer to C. Gerhardt \cite[Lemma~2.1.14]{cg:curvature}.
\end{proof}

\begin{lemma}[Linear operator at a critical point]\label{lem3:critical point}
  Let $w=w\big(h^j_i\big)$ be a symmetric function of the principal curvatures $a$, $b$, and $c$.
  At a critical point of $w$, \ie $w_{;i}=0$ for all $i=1,2,3$, we choose normal coordinates, \ie
  $g_{ij}=\delta_{ij}$ and $\big(h_{ij}\big)=\diag(a,b,c)$.
  Then we have
  \begin{align}
    \begin{split} 
      Lw =&\, \vec{C_w}(a,b,c) \\
          & + \vec{E_w}(a,b,c)\,h_{12;3}^2 \\
          & + \vec{R_w}(a,b,c,h_{11;1},h_{22;1},h_{33;1}) \\
          & + \vec{S_w}(a,b,c,h_{11;2},h_{22;2},h_{33;2}) \\
          & + \vec{T_w}(a,b,c,h_{11;3},h_{22;3},h_{33;3})
    \end{split} 
  \end{align}
  The constant terms $\vec{C_w}$ are
  \begin{align*}
    \vec{C_w}(a,b,c) =&\, a\,w_a\left(a^2\,F_a+b^2\,F_b+c^2\,F_c+a\left(F-a\,F_a-b\,F_b-c\,F_c\right)\right) \umbruch \\
                & +b\,w_a\left(a^2\,F_a+b^2\,F_b+c^2\,F_c+b\left(F-a\,F_a-b\,F_b-c\,F_c\right)\right) \umbruch \\
                & +c\,w_c\left(a^2\,F_a+b^2\,F_b+c^2\,F_c+c\left(F-a\,F_a-b\,F_b-c\,F_c\right)\right).
  \end{align*}
  The gradient terms $\vec{E_w}$ are
  \begin{align*}
    \vec{E_w}(a,b,c)/2=&\, \Big( \left(w_c\left(F_a-F_b\right)-F_c\left(w_a-w_b\right)\right)/(a-b) \Big. \umbruch \\
                  & \Big. + \left(w_b\left(F_a-F_c\right)-F_b\left(w_a-w_c\right)\right)/(a-c) \Big. \umbruch \\
                  & \Big. + \left(w_a\left(F_b-F_c\right)-F_a\left(w_b-w_c\right)\right)/(b-c) \Big).
  \end{align*}
  The gradient terms $\vec{R_w}$ are
  \begin{align*}
      \vec{R_w}(a,b,c,&\, h_{11;1},h_{22;1},h_{33;1}) \umbruch \\
     =&\, w_a\left(\left(F_{aa}\,h_{11;1}^2+F_{bb}\,h_{22;1}^2+F_{cc}\,h_{33;1}^2\right) \right. \\
      &\qquad\quad \left. +2\left(F_{ab}\,h_{11;1}\,h_{22;1} + F_{ac}\,h_{11;1}\,h_{33;1} + F_{bc}\,h_{22;1}\,h_{33;1} \right)\right) \umbruch \\
      &\,+w_b\left(2\frac{F_a-F_b}{a-b} h_{22;1}^2 \right) \umbruch \\
      &\,+w_c\left(2\frac{F_a-F_c}{a-c} h_{33;1}^2 \right) \umbruch \\
      &\,-F_a\left(w_{aa}\,h_{11;1}^2+w_{bb}\,h_{22;1}^2 + w_{cc}\,h_{33;1}^2 \right. \\
      &\qquad\quad \left.+2\left(w_{ab}\,h_{11;1}\,h_{22;1} + w_{ac}\,h_{11;1}\,h_{33;1} + w_{bc}\,h_{22;1}\,h_{33;1} \right)\right) \umbruch \\
      &\,-F_b\left(2\frac{w_a-w_b}{a-b} h_{22;1}^2\right) \\
      &\,-F_c\left(2\frac{w_a-w_c}{a-c} h_{33;1}^2\right).                      
  \end{align*}
  The gradient terms $\vec{S_w}$ are
  \begin{align*}
      \vec{S_w}(a,b,c,&\,h_{11;2},h_{22;2},h_{33;2}) \\
     =&\,w_a\left(2\frac{F_a-F_b}{a-b} h_{11;2}^2 \right) \\
      &\,+w_b\left(F_{aa}\,h_{11;2}^2 + F_{bb}\,h_{22;2}^2 + F_{cc}\,h_{33;2}^2 \right. \\
      &\qquad\quad \left.+2\left(F_{ab}\,h_{11;2}\,h_{22;2} + F_{ac}\,h_{11;2}\,h_{33;2} + F_{bc}\,h_{22;2}\,h_{33;2} \right)\right) \\
      &\,+w_c\left(2\frac{F_b-F_c}{b-c} h_{33;2}^2 \right) \\
      &\,-F_a\left(2\frac{w_a-w_b}{a-b} h_{11;2}^2 \right) \\
      &\,-F_b\left(w_{aa}\,h_{11;2}^2+w_{bb}\,h_{22;2}^2 + w_{cc}\,h_{33;2}^2 \right. \\
      &\qquad\quad \left. +2\left(w_{ab}\,h_{11;2}\,h_{22;2}+w_{ac}\,h_{11;2}\,h_{33;2}+w_{bc}\,h_{22;2}\,h_{33;2}\right)\right) \\
      &\,-F_c\left(2\frac{w_b-w_c}{b-c} h_{33;2}^2\right).
  \end{align*}
  The gradient terms $\vec{T_w}$ are
  \begin{align*}
       \vec{T_w}(a,b,c,&\,h_{11;3},h_{22;3},h_{33;3}) \umbruch \\
     =&\,w_a\left(2\frac{F_a-F_c}{a-c} h_{11;3}^2 \right) \umbruch \\
      &\,+w_b\left(2\frac{F_b-F_c}{b-c} h_{22;3}^2 \right) \umbruch \\
      &\,+w_c\left(F_{aa}\,h_{11;3}^2+F_{bb}\,h_{22;3}^2+F_{cc}\,h_{33;3}^2 \right. \\
      &\qquad\quad \left.+2\left(F_{ab}\,h_{11;3}\,h_{22;3} + F_{ac}\,h_{11;3}\,h_{33;3}+F_{bc}\,h_{22;3}\,h_{33;3}\right)\right) \umbruch \\
      &\,-F_a\left(2\frac{w_a-w_c}{a-c} h_{11;3}^2 \right) \umbruch \\
      &\,-F_b\left(2\frac{w_b-w_c}{b-c} h_{22;3}^2 \right) \umbruch \\
      &\,-F_c\left(w_{aa}\,h_{11;3}^2 + w_{bb}\,h_{22;3}^2 + w_{cc}\,h_{33;3}^2 \right. \\
      &\qquad\quad \left.+2\left(w_{ab}\,h_{11;3}\,h_{22;3} + w_{ac}\,h_{11;3}\,h_{33;3} + w_{bc}\,h_{22;3}\,h_{33;3} \right)\right).
  \end{align*}  
  Furthermore, we have at a critical point of $w$
  \begin{align}\label{id:critical point}
    \begin{split}
      h_{11;1} =&\, -\frac{1}{w_a}\left(w_b\,h_{22;1} + w_c\,h_{33;1} \right), \\
      h_{22;2} =&\, -\frac{1}{w_b}\left(w_a\,h_{11;2} + w_c\,h_{33;2} \right), \\
      h_{33;3} =&\, -\frac{1}{w_c}\left(w_a\,h_{11;3} + w_b\,h_{22;3} \right).
    \end{split}
  \end{align}
\end{lemma}
\begin{proof}
  We use Lemma \ref{lem3:linear operator}, and Lemma \ref{lem3:second derivatives} at a point,
  where we choose normal coordinates.
  This way we obtain the constant terms $\vec{C_w}$, and the four gradient terms $\vec{E_w}$, $\vec{R_w}$, $\vec{S_w}$, and $\vec{T_w}$. \\
  
  At a critical point of $w$, we have $w_i(a,b,c) = 0$ for $i=1,2,3$.
  This implies
  \begin{align*}
    w_a\,h_{1l;i}\,g^{l1} + w_b\,h_{2l;i}\,g^{l2} + w_c\,h_{3l;i}\,g^{l3} = 0.
  \end{align*}
  Using normal coordinates we obtain
  \begin{align*}
    w_a\,h_{11;i} + w_b\,h_{22;i} + w_c\,h_{33;i} = 0.
  \end{align*}
  Now we obtain for $i=1,2,3$ the identities
  \begin{align*}
    h_{11;1} =&\, -\frac{1}{w_a}\left(w_b\,h_{22;1} + w_c\,h_{33;1} \right), \\
    h_{22;2} =&\, -\frac{1}{w_b}\left(w_a\,h_{11;2} + w_c\,h_{33;2} \right), \\
    h_{33;3} =&\, -\frac{1}{w_c}\left(w_a\,h_{11;3} + w_b\,h_{22;3} \right).
  \end{align*}
  This concludes the proof.
\end{proof}

\begin{corollary}[Linear operator at a critical point]\label{cor3:critical point}
  Let the gradient terms $\vec{R_w}$, $\vec{S_w}$, and $\vec{T_w}$
  be defined as in Lemma \ref{lem3:critical point}. \\
  Then we have
  \begin{align*}
    \vec{R_w}(a,b,c,h_{22;1},h_{33;1}) =&\, \binom{h_{22;1}}{h_{33;1}}^\top\, M^\vec{R_w}(a,b,c)\, \binom{h_{22;1}}{h_{33;1}}, \\
    \vec{S_w}(a,b,c,h_{11;2},h_{33;2}) =&\, \binom{h_{11;2}}{h_{33;2}}^\top\, M^\vec{S_w}(a,b,c)\, \binom{h_{11;2}}{h_{33;2}}, \\
    \vec{T_w}(a,b,c,h_{11;3},h_{22;3}) =&\, \binom{h_{11;3}}{h_{22;3}}^\top\, M^\vec{T_w}(a,b,c)\, \binom{h_{11;3}}{h_{22;3}}.
  \end{align*}
  The elements of the matrix $M^\vec{R_w}(a,b,c)$ are
  \begin{align*}
    m_{11}^\vec{R_w}(a,b,c) =&\, 2\frac{F_a\,w_b-F_b\,w_a}{a-b} \\
                             &\, + F_{aa} \frac{w_b^2}{w_a} - 2 F_{ab}\,w_b + F_{bb}\,w_a \\
                             &\, - F_a \left( w_{aa} \frac{w_b^2}{w_a^2} - 2 w_{ab} \frac{w_b}{w_a} + w_{bb} \right), \umbruch \\
    m_{12}^\vec{R_w}(a,b,c) =&\, F_{aa} \frac{w_b\,w_c}{w_a} - F_{ab}\,w_c - F_{ac}\,w_b + F_{bc}\,w_a \\
                             &\, -F_a \left( w_{aa} \frac{w_b\,w_c}{w_a^2} - w_{ab} \frac{w_c}{w_a} - w_{ac} \frac{w_b}{w_a} + w_{bc} \right), \umbruch \\
    m_{22}^\vec{R_w}(a,b,c) =&\, 2\frac{F_a\,w_c - F_c\,w_a}{a-c} \\
                             &\, + F_{aa}\,\frac{w_c^2}{w_a} -2\,F_{ac}\,w_c +F_{cc}\,w_a \\
                             &\, -F_a\left(w_{aa} \frac{w_c^2}{w_a^2} -2 w_{ac}\,\frac{w_c}{w_a} + w_{cc} \right).
  \end{align*}
  The elements of the matrix $M^\vec{S_w}(a,b,c)$ are
  \begin{align*}
    m_{11}^\vec{S_w}(a,b,c) =&\, 2\frac{F_a\,w_b-F_b\,w_a}{a-b} \\
                             &\, + F_{aa}\,w_b -2 F_{ab}\,w_a + F_{bb} \frac{w_a^2}{w_b} \\
                             &\, -F_b \left(w_{aa} - 2 w_{ab} \frac{w_a}{w_b} + w_{bb} \frac{w_a^2}{w_b^2} \right), \umbruch \\
    m_{12}^\vec{S_w}(a,b,c) =&\, -F_{ab}\,w_c + F_{ac}\,w_b + F_{bb} \frac{w_a\,w_c}{w_b} - F_{bc}\,w_a \\
                             &\, -F_b \left(-w_{ab} \frac{w_c}{w_b} + w_{ac} + w_{bb} \frac{w_a\,w_c}{w_b^2} - w_{bc} \frac{w_a}{w_b} \right), \umbruch \\
    m_{22}^\vec{S_w}(a,b,c) =&\, 2\frac{F_b\,w_c-F_c\,w_b}{b-c} \\
                             &\, +F_{bb} \frac{w_c^2}{w_b} - 2 F_{bc}\,w_c + F_{cc}\,w_b \\
                             &\, -F_b \left( w_{bb} \frac{w_c^2}{w_b^2} - 2 w_{bc} \frac{w_c}{w_b} + w_{cc} \right).                  
  \end{align*}
  The elements of the matrix $M^\vec{T_w}(a,b,c)$ are
  \begin{align*}
    m_{11}^\vec{T_w}(a,b,c) =&\, 2 \frac{F_a\,w_c-F_c\,w_a}{a-c} \\
                             &\, +F_{aa}\,w_c - 2 F_{ac}\,w_a + F_{cc} \frac{w_a^2}{w_c} \\
                             &\, -F_c \left( w_{aa} - 2 w_{ac} \frac{w_a}{w_c} + w_{cc} \frac{w_a^2}{w_c^2} \right), \umbruch \\
    m_{12}^\vec{T_w}(a,b,c) =&\, F_{ab}\,w_c - F_{ac}\,w_b - F_{bc}\,w_a + F_{cc} \frac{w_a\,w_b}{w_c} \\
                             &\, -F_c\left( w_{ab} - w_{ac} \frac{w_b}{w_c} - w_{bc} \frac{w_a}{w_c} + w_{cc} \frac{w_a\,w_b}{w_c^2} \right), \umbruch \\
    m_{22}^\vec{T_w}(a,b,c) =&\, 2 \frac{F_b\,w_c - F_c\,w_b}{b-c} \\
                             &\, +F_{bb}\,w_c - 2 F_{bc}\,w_b + F_{cc}\,\frac{w_b^2}{w_c} \\
                             &\, -F_c\left(w_{bb} - 2 w_{bc} \frac{w_b}{w_c} + w_{cc} \frac{w_b^2}{w_c^2} \right).
  \end{align*}
\end{corollary}
\begin{proof}
  We use identities \eqref{id:critical point} to replace $h_{11;1}$, $h_{22;2}$, and $h_{33;3}$ in $\vec{R_w}$, $\vec{S_w}$, and $\vec{T_w}$ 
  from Lemma \ref{lem3:critical point}, respectively.
  Now we rewrite the quadratic forms $\vec{R_w}$, $\vec{S_w}$, and $\vec{T_w}$ as $\vec{x}^\top\,M\,\vec{x}$.
  This concludes the proof.
\end{proof}

\begin{remark}[Sufficient conditions for the non-positivity of the linear operator]\label{rem3:sufficient conditions}
  Under the assumptions of Lemma \ref{lem3:critical point}, the linear operator $Lw$ is non-positive at some critical point of $w$, 
  if $\vec{C_w}$, $\vec{E_w}$ are non-positive, and $M^\vec{R_w}$, $M^\vec{S_w}$, and $M^\vec{T_w}$ are negative semi-definite there.
  This is a direct consequence of Lemma \ref{lem3:critical point}, and Corollary \ref{cor3:critical point}. \\

  Now let $M \in \R^{2\times 2}$ be a symmetric matrix. 
  Then we have the equivalent conditions
  \begin{enumerate}
    \item $M=\begin{pmatrix} m_{11} & m_{12} \\ m_{12} & m_{22} \end{pmatrix}$ is negative semi-definite,
    \item $\tr M=m_{11}+m_{22} \leq 0$, and $-\det M=m_{12}^2-m_{11}\,m_{22} \leq 0$.
  \end{enumerate}
  $\left. \right.$ \\
  In $[\textsc{cp}]$, we check the non-positivity of the linear operator $Lw$ 
  by checking the non-positivity of $\vec{C_w}$, $\vec{E_w}$, and by checking 
  condition $(2)$ for the matrices $M^\vec{R_w}$, $M^\vec{S_w}$, and $M^\vec{T_w}$.
\end{remark}

\section{Vanishing functions}

In $\R^{n+1}$, $n=2$, for many normal velocities $F$ the quantity
\begin{align*}
  \frac{\left(\lambda_1-\lambda_2\right)^2}{\left(\lambda_1\,\lambda_2\right)^2} F^2
\end{align*}
seems to be the natural choice when showing convergence to a round point.
As in \cite{mf:when}, we call this quantity a \textit{vanishing function} for a normal velocity $F$.
It is used by B. Andrews for the Gauss curvature flow \cite{ba:gauss}, 
by F. Schulze and O. Schn\"urer for the $H^\sigma$-flow \cite{fs:convexity}, 
by B. Andrews and X. Chen for the $|A|^\sigma$ and the $\tr A^\sigma$-flow \cite{ac:surfaces}. 

First we give the Definition \ref{def3:vanishing function} of a vanishing function in $\R^{n+1}$, $n=3$.
In Remark \ref{rem3:vanishing function} we then introduce 
\begin{align*}
  \sum_{i<j} \frac{\left(\lambda_i-\lambda_j\right)^2}{\left(\lambda_i\,\lambda_j\right)^2} F^2
\end{align*}
as the counterpart of a vanishing function for arbitrary dimensions $n\in\N$.
In this paper we work with the quantity in particular in $\R^{n+1}$, $n=3$.

In Lemma \ref{lem3:vanishing function} we deduce a simple but interesting estimate 
for vanishing functions for arbitrary dimensions $n\in\N$.
We employ this Lemma \ref{lem3:vanishing function}
in the proof of our main Theorem \ref{thm3:H3}.

\begin{definition}[Vanishing function]\label{def3:vanishing function}
  Let $v(a,b,c)\in C^2\left(\R^3_{+}\right)$ with $v\not\equiv 0$.
  Let $\vec{C_w}(a,b,c)$ be defined as in Lemma \ref{lem3:critical point}.
  We call $v$ a \textit{vanishing function} for a normal velocity $F$
  if $\vec{C_v}(a,b,c) = 0$ for all $0<a,b,c$.
\end{definition}

\begin{example}[Vanishing function]\label{exm3:vanishing function}
  We have the following example of a vanishing function for a normal velocity $F$:
  \begin{align*}
    \left( \frac{(a-b)^2}{(a\,b)^2} + \frac{(a-c)^2}{(a\,c)^2} + \frac{(b-c)^2}{(b\,c)^2} \right) F^2.
  \end{align*}
\end{example}

\begin{remark}[Vanishing function]\label{rem3:vanishing function}
  We can define a \textit{vanishing function} for arbitrary dimensions $n\in\N$.
  Let $\lambda_i$, $i=1,\ldots,n$, denote the principal curvatures
  of a hypersurface in $\R^{n+1}$.
  Using Lemma \ref{lem3:linear operator} we can define 
  constant terms $\vec{C_w}(\lambda_1,\ldots,\lambda_n)$ 
  as in Lemma \ref{lem3:critical point} for an arbitrary $n\in\N$.
  We have the following example of a vanishing function:
  \begin{align}\label{id3:vanishing function}
    \sum_{i<j} \frac{\left(\lambda_i-\lambda_j\right)^2}{\left(\lambda_i\,\lambda_j\right)^2} F^2.
  \end{align}
  Interestingly, we still obtain a vanishing function if we omit up to $n-1$ terms of the form
  \begin{align*}
    \frac{\left(\lambda_i-\lambda_j\right)^2}{\left(\lambda_i\,\lambda_j\right)^2} F^2.
  \end{align*}
  This reminds us of \cite[Theorem 1.5]{hs:mean} by G. Huisken and C. Sinestrari.
\end{remark}

\begin{lemma}[Vanishing function]\label{lem3:vanishing function}
  Let $v$ be a vanishing function as defined in Remark \ref{rem3:vanishing function}.
  Let $v\leq C^2$ on some set $\Sc\subset\R^n_{+}$, and for some constant $C>0$. \\
  Then we have
  \begin{align}
    1 \leq \frac{\lambda_{\text{max}}}{\lambda_{\text{min}}} \leq 1 + C \frac{\lambda_{\text{max}}}{F} \qquad \text{on}\;\;\Sc.
  \end{align}
\end{lemma}
\begin{proof}
  We assume $\lambda_{\text{min}} \equiv \lambda_1 \leq \ldots \leq \lambda_n \equiv \lambda_{\max}$ and obtain
  \begin{align*}
    C^2 \geq \sum_{i<j} \frac{\left(\lambda_i-\lambda_j\right)^2}{\left(\lambda_i\,\lambda_j\right)^2} F^2
        \geq \frac{\left(\lambda_n-\lambda_1\right)^2}{\left(\lambda_1\,\lambda_n\right)^2} F^2,
  \end{align*}
  which implies
  \begin{align*}
    C \frac{\lambda_n}{F} \geq \frac{\lambda_n\left(\lambda_n-\lambda_1\right)}{\lambda_1\,\lambda_n} = \frac{\lambda_n}{\lambda_1}-1\geq 0.
  \end{align*}
  This concludes the proof.
\end{proof}

\section{$H^3$-flow}

The proof of our main Theorem \ref{thm3:H3} is based on investigating
\begin{align*}
  \phi_{H^3} =&\, \frac{(a-b)^2+(a-c)^2+(b-c)^2}{(a+b+c)^2}, \\
  \text{and }\quad \psi_{H^3} =&\,  \left( \frac{(a-b)^2}{(a\,b)^2} + \frac{(a-c)^2}{(a\,c)^2} + \frac{(b-c)^2}{(b\,c)^2} \right) \left(H^3\right)^2.
\end{align*}
The quantity $\phi_{H^3}$ is inspired by the quantity used in \cite{gh:flow} by G. Huisken.
The other quantity $\psi_{H^3}$ is a vanishing function.
First we show that the estimate $\phi_{H^3}\leq h := 1/8$ is preserved during the $H^3$-flow  
if the initial hypersurface $M_0$ is $2$-pinched. This is Lemma \ref{lem3:phi} and Corollary \ref{cor3:phi}.
Next we prove that $\psi_{H^3}$ is bounded in time on the set where $\phi_{H^3}\leq h$.
This is Lemma \ref{lem3:psi} and Corollary \ref{cor3:psi}.

The proofs of these Lemmas and Corollaries involve the maximum-principle, the linear operator $L$ and our computer program $[\textsc{cp}]$.
In $[\textsc{cp}]$ we deal with computations of two kinds. 
One kind is the purely algebraic manipulation of terms, and could still be performed by pen and paper.
The other kind of computations includes random numbers for a Monte-Carlo method, which appears to be very tedious to carry out with pen and paper.

Finally, we show convergence to a round point combining the boundedness of $\psi_{H^3}$ and the proof of \cite[Theorem A.1.]{fs:convexity} 
by F. Schulze and O. Schn\"urer. This is our main Theorem \ref{thm3:H3}.

\begin{lemma}[Monotone quantity $\phi$]\label{lem3:phi}
  Let $\left(M_t\right)_{0\leq t < T}$ be a maximal solution of the $H^3$-flow, 
  where $M_0$ is $2$-pinched. 
  Then we have 
  \begin{align*}
    L\phi \leq 0
  \end{align*}
  at a critical point of $\phi$, where $0<\phi \leq 1/8 =: h$.
\end{lemma}
\begin{proof}
  Let $\Sc_{C_2}$ be the $2$-pinched cone in the positive orthant.
  In $[\textsc{cp}]$, we compute
  \begin{align*}
    \Sc_{C_2} :=&\, \{(\lambda_1,\lambda_2,\lambda_3) \in \R^3_{+} : \lambda_i/\lambda_j \leq 2 \text{ for all } 1\leq i,j\leq 3 \} \\
    \Sc_h :=&\, \{(a,b,c) \in \R^3_{+} : 0 <\phi \leq 1/8 = h \}, \\
    \Sc_{L\phi} :=&\, \{(a,b,c) \in \R^3_{+} : L\phi \leq 0 \}
  \end{align*}
  and show that
  \begin{align*}
    \Sc_{C_2} \subset \Sc_h \subset \Sc_{L\phi}
  \end{align*}
  using a Monte-Carlo method.
  Here, we check the non-positivity of $L\phi$ as described in Remark \ref{rem3:sufficient conditions}. \\
  Since the functions $\lambda_i/\lambda_j$, $\phi$, and $L\phi$ are homogeneous in the principal curvatures,
  it suffices to compute the sets $\Sc_{C_2}$, $\Sc_h$, and $\Sc_{L\phi}$ in $[\textsc{cp}]$
  for the radial projection
  \begin{align*}
    \pi: \R^3_{+} \to \{a+b+c=1\},\quad(a,b,c) \mapsto (a,b,c)/(a+b+c).
  \end{align*}
  This concludes the computer-based proof.  
\end{proof}

\begin{corollary}[Monotone quantity $\phi$]\label{cor3:phi}
  Let $\left(M_t\right)_{0\leq t < T}$ be a maximal solution of the $H^3$-flow, 
  where $M_0$ is $2$-pinched. 
  Then we have that 
  \begin{align*}
    \phi \leq h := 1/8
  \end{align*} 
  during the $H^3$-flow.
\end{corollary}
\begin{proof}
  This follows directly from Lemma \ref{lem3:phi} using the maximum-principle.
\end{proof}

\begin{lemma}[Monotone quantity $\psi$]\label{lem3:psi}
  Let $\left(M_t\right)_{0\leq t < T}$ be a maximal solution of the $H^3$-flow, 
  where $M_0$ is $2$-pinched. 
  Then we have
  \begin{align*}
    L\psi \leq 0
  \end{align*}
  at a critical point of $\psi$, where $\psi>0$.
\end{lemma}
\begin{proof}
  By Corollary \ref{cor3:phi}, we have
  \begin{align*}
    \phi \leq h
  \end{align*}
  during the $H^3$-flow.
  Let $\Sc_{C_2}$, $\Sc_h$ be defined as in Lemma \ref{lem3:phi}. 
  In $[\textsc{cp}]$, we also compute 
  \begin{align*}
    \Sc_{L\psi} := \{(a,b,c) \in \R^3_{+} : L\psi \leq 0 \}
  \end{align*}
  and show in particular the second inclusion of
  \begin{align*}
    \Sc_{C_2} \subset \Sc_h \subset \Sc_{L\psi}
  \end{align*}
  using a Monte-Carlo method. By Lemma \ref{lem3:phi}, we have the first conclusion.
  
  This concludes the computer-based proof.
\end{proof}

\begin{corollary}[Monotone quantity $\psi$]\label{cor3:psi}
  Let $\left(M_t\right)_{0\leq t < T}$ be a maximal solution of the $H^3$-flow, 
  where $M_0$ is $2$-pinched. 
  Then we have that 
  \begin{align*}
    \max_{M_t}\,\psi
  \end{align*}
  is non-increasing during the $H^3$-flow.
\end{corollary}
\begin{proof}
  This follows directly from Lemma \ref{lem3:psi} using the maximum-principle.
\end{proof}

\begin{theorem}[$H^3$-flow]\label{thm3:H3}
 Let $\left(M_t\right)_{0\leq t < T}$ be a maximal solution of the $H^3$-flow, 
 where $M_0$ is $2$-pinched. 
 Then $\left(M_t\right)_{0\leq t < T}$ converges to a round point.
\end{theorem}
\begin{proof}
  We closely follow proof of the corresponding \cite[Theorem A.1.]{fs:convexity}
  by F. Schulze and O. Schn\"urer. \\
  
  By \cite[Theorem 1.1]{fs:convexity} the surfaces $M_t$ become immediately 
  strictly convex for $t>0$. Now choose a sufficiently small $0<\epsilon<T$
  such that the $H^3$-flow is smooth and strictly convex on the interval
  $(\epsilon,T)$.
  Thus the quantity $\psi_{H^3}$ is well-defined on this interval,
  and bounded from above by Corollary \ref{cor3:psi}.
  By Lemma \ref{lem3:vanishing function} this implies
  \begin{align*}
    1 \leq \frac{\lambda_{\text{max}}}{\lambda_{\text{min}}} \leq 1 + \frac{C}{H^2}
  \end{align*}
  on $\left(M_t\right)_{(\epsilon<t<T)}$. 
  
  Now the proof follows analogously to the proof of \cite[Theorem 1.2]{fs:convexity}.
\end{proof}

\section{$|A|^2$-flow}

A result similar to our main Theorem \ref{thm3:H3} holds for the normal velocity $F=|A|^2$ and $3$-pinched hypersurfaces. For a proof consider
\begin{align*}
  \phi_{|A|^2} =&\, \frac{(a^2+b^2+c^2)(a\,b+a\,c+b\,c)^2}{(a\,b\,c)^2} \\
  \text{and}\quad \psi_{|A|^2} =&\, \frac{(a+b+c)^2\left( (a-b)^2 + (a-c)^2 + (b-c)^2 \right)}{a\,b\,c},
\end{align*}
and O. Schn\"urer \cite{os:surfacesA2}.
As in chapter on $H^3$-flow using $[\textsc{cp}]$ we obtain
\begin{lemma}[$|A|^2$-flow]\label{lem3:A2}
  Let $\left(M_t\right)_{0\leq t < T}$ be a maximal solution of the $|A|^2$-flow, 
  where $M_0$ is $3$-pinched. 
  Then we have that 
  \begin{align*}
    \max_{M_t}\,\psi_{|A|^2}
  \end{align*}
  is non-increasing in time.
\end{lemma}

\section{Gauss curvature flow}

A result similar to our main Theorem \ref{thm3:H3} holds for the normal velocity $F=K$ and $2$-pinched hypersurfaces. For a proof consider
\begin{align*}
  \phi_K =&\, \frac{(a-b)^2 + (a-c)^2 + (b-c)^2}{a^2+b^2+c^2} \umbruch \\
  \text{and}\quad \psi_K =&\, \left( \frac{(a-b)^2}{(a\,b)^2} + \frac{(a-c)^2}{(a\,c)^2} + \frac{(b-c)^2}{(b\,c)^2} \right) \left(K\right)^2,
\end{align*}
and B. Chow \cite{bc:deforming}.
As in chapter on $H^3$-flow using $[\textsc{cp}]$ we obtain
\begin{lemma}[Gauss curvature flow]\label{lem3:K}
  Let $\left(M_t\right)_{0\leq t < T}$ be a maximal solution of the Gauss curvature flow, 
  where $M_0$ is $2$-pinched. 
  Then we have that 
  \begin{align*}
    \max_{M_t}\,\psi_K
  \end{align*}
  is non-increasing in time.
\end{lemma}

\section{Outlook}

Our aim is to show convergence to a round point without pinching requirements 
using vanishing functions in arbitrary dimensions.
Instead of splitting the linear operator $L$ into constant terms and gradient terms
we intend to work with integral estimates similar to G. Huisken \cite{gh:flow}.
This way we seek to prove convergence to a round point for contracting normal velocities, 
including powers of the Gauss curvature, $K$, of the mean curvature, $H$, 
and of the norm of the second fundamental form, $|A|$.

\section{Appendix}

Some of the Lemmas leading up to the proof our main Theorem \ref{thm3:H3} rely
on the computer program $[\textsc{cp}]$. First we compute the linear operator $L$ 
for the corresponding quantities $\phi$ and $\psi$.
Next we use a Monte-Carlo method to compute the sets $\Sc_{C_2}$, $\Sc_h$, $\Sc_{L\phi}$, and $\Sc_{L\psi}$.
Finally, we compute the two inclusions
\begin{align}\label{id:inclusions}
  \begin{split}
    \Sc_{C_2} \subset&\, \Sc_h \subset \Sc_{L\phi}, \\
    \Sc_{C_2} \subset&\, \Sc_h \subset \Sc_{L\psi}. 
  \end{split}
\end{align}

For the convenience of the reader, we include the source code of $[\textsc{cp}]$ 
in three different programming languages, namely for the computer algebra systems Mathematica, Sage, and Maple.
The first part of the appendix is the Mathematica program, the second part is the Sage program, and the third part is the Maple program.

In the first part we also visualize the two inclusions \eqref{id:inclusions}. \\

At \url{www.arxiv.org} we can only submit this article without 
the computer program $[\textsc{cp}]$. 
To download this article with the computer program $[\textsc{cp}]$ please go to
\url{www.martinfranzen.de}.

\bibliographystyle{amsplain} 
\def\weg#1{} \def\unterstrich{\underline{\rule{1ex}{0ex}}} \def\cprime{$'$}
  \def\cprime{$'$} \def\cprime{$'$} \def\cprime{$'$}
\providecommand{\bysame}{\leavevmode\hbox to3em{\hrulefill}\thinspace}
\providecommand{\MR}{\relax\ifhmode\unskip\space\fi MR }
\providecommand{\MRhref}[2]{%
  \href{http://www.ams.org/mathscinet-getitem?mr=#1}{#2}
}
\providecommand{\href}[2]{#2}

\end{document}